\documentclass[12pt]{amsart}

\usepackage{
amsfonts,
latexsym,
amssymb,
mathabx,
}

\usepackage{enumerate}

\makeatletter
\@namedef{subjclassname@2010}{%
  \textup{2010} Mathematics Subject Classification}
\makeatother

 \changenotsign

\newcommand{\labbel}{\label}

\newtheorem{theorem}{Theorem}

\newtheorem{thm}[theorem]{Theorem}

\newtheorem{proposition}[theorem]{Proposition} 
 
\newtheorem{corollary}[theorem]{Corollary}

\theoremstyle{definition}
\newtheorem{definition}[theorem]{Definition}

\theoremstyle{remark}

\newcommand{\brfrt}{\hspace{0 pt}}

\newcommand{\m}{\mathfrak}

\begin{document}

\baselineskip=17pt

\title{Compactness of powers of $ \omega$}

\author[P. Lipparini]{Paolo Lipparini} 
\address{Dimartipento di Matematica\\Viale della Ricerca Scientifica\\II Universit\`a di Roma (Tor Vergata)\\I-00133 ROME ITALY}
\urladdr{http://www.mat.uniroma2.it/\textasciitilde lipparin}



\begin{abstract}
We characterize exactly the compactness properties of
the 
product of $\kappa$ copies of the space $ \omega$ with the discrete topology.
The characterization involves uniform ultrafilters,
infinitary languages, and
the existence of nonstandard elements
in elementary estensions.
We also have results involving 
products of possibly uncountable  regular cardinals.
\end{abstract}

\maketitle

Mycielski \cite{M}, extending previous results by Ehrenfeucht, 
Erd\"os, Hajnal, \L o\'s  and Stone,
showed that $\omega^ \kappa $ is not (finally) $\kappa$-compact, 
for every infinite cardinal $\kappa$ strictly less than the first weakly inaccessible
cardinal. Here $\omega$ denotes a countable topological space with the discrete topology; products (and powers) are endowed with the Tychonoff topology, and a topological space is said to be \emph{finally $\kappa$-compact} if any open cover has a subcover of
cardinality strictly less than $\kappa$. 

On the other direction, Mr\'owka \cite{
Mr1,
Mr2}
showed that if 
$\mathcal L _{ \omega _1, \omega } $
is $ ( \kappa  , \kappa   )$-compact, then
$\omega^ \kappa $ is indeed finally $\kappa$-compact
(in particular, this holds if $\kappa$ is weakly compact).
As usual, $\mathcal L _{ \lambda , \mu  } $ is 
the \emph{infinitary language} which 
allows
conjunctions 
and disjunctions 
of $ < \lambda $ formulas,  and universal or  existential quantification 
over $<\mu$ variables;
\emph{$( \kappa  , \kappa  )$-\brfrt compactness}
means that any $ \kappa $-satisfiable set of $|\kappa|$-many sentences
is satisfiable.

To the best of our knowledge, the gap between 
Mycielski's and  Mr\'owka's results has never been 
exactly filled.
It follows from 
\cite[Theorem 1]{Mr2}  and 
\v{C}udnovski\u{\i} \cite[Theorem 2]{Cho} 
 that 
$\mathcal L _{ \kappa , \omega } $
is $ ( \kappa  , \kappa   )$-compact
if and only if 
every product of $| \kappa |$-many discrete spaces,
each of cardinality $<\kappa$, is finally 
$\kappa$-compact
(the proofs build also on work by 
Hanf, Keisler, Monk, Scott, Tarski, Ulam and others;
earlier versions and variants were known under inaccessibility conditions).  
No matter how satisfying the above  result is, it
adds nothing about powers of $ \omega$, since it deals
with possibly uncountable factors. 

In this note we  show that
   Mr\'owka gives the   exact estimation, namely, that
$\omega^ \kappa $ is  finally $\kappa$-compact
if and only if 
$\mathcal L _{ \omega _1, \omega } $
is $ ( \kappa  , \kappa   )$-compact.
More generally, we find necessary and sufficient conditions
for $\omega^ \kappa $ being finally $ \lambda  $-compact,
or, even, just being $[ \lambda , \lambda ]$-compact. 
Our methods involve intermediate steps of independent interest,
dealing with uniform ultrafilters
and extensions of models by means of ``$\lambda$-nonstandard'' elements.
The equivalences we find in such intermediate steps hold
for arbitrary regular cardinals, not only for $ \omega$; in particular,
compactness properties of products of regular cardinals (with the order topology)
are characterized. 

Throughout, $\lambda$, $\mu$, $\kappa$ and $\nu$ are infinite cardinals, 
$X$ is a topological space, and $D$ is an ultrafilter. Cardinals are 
also considered as topological spaces 
endowed with the order topology.

The space $X$ is
\emph{$[\mu, \lambda ]$-compact} if  every open 
cover of $X$ by at most $\lambda$  sets has a subcover by less than $\mu$  sets.
It is easy to show that final $\kappa$-compactness is equivalent to 
$[ \nu , \nu ]$-compactness, for every 
$\nu \geq \kappa $, or, more generally,
that $[\mu, \lambda ]$-compactness is equivalent to 
$[ \nu , \nu ]$-compactness, for every 
$\nu $ such that $\mu \leq \nu \leq \lambda $.
If $D$ is an ultrafilter over some set $I$, a sequence 
$(x _ i) _{i \in I } $ of elements of $X$ 
is said to \emph{$D$-converge} 
to $x \in X$ if 
 $\{ i \in I \mid x_i \in U\} \in D$,
for every open neighborhood $U$ of $x$.
If $f : I \to J$ is a function, $f(D)$ is the ultrafilter over $J$ 
defined by $Y \in f(D)$ if and only if $ f ^{-1}(Y) \in D $.

\begin{definition} \labbel{def}    
We shall denote by
$\lambda {\not\Rightarrow} (\mu_ \gamma ) _{ \gamma \in \kappa } $ 
the following statement.
 \begin{enumerate}   
\item[(*)]
For every sequence of functions 
$(f_ \gamma  ) _{ \gamma \in \kappa }$,
such that  $f_ \gamma   : \lambda \to \mu_ \gamma  $
for $\gamma \in \kappa$,
there is some uniform ultrafilter $D$ over $\lambda$ such that,
for no $\gamma \in \kappa $,
$f_ \gamma (D)$ is uniform over $\mu_ \gamma $.    
 \end{enumerate}

We shall write 
$\lambda \stackrel{ \kappa }{\not\Rightarrow} \mu$ 
when all the $\mu_ \gamma $'s in (*) are equal to $\mu$.

The negation of  $\lambda \stackrel{ \kappa }{\not\Rightarrow} \mu$  is denoted by $\lambda \stackrel{ \kappa }{\Rightarrow} \mu$. 
\end{definition}

The following observation by Saks \cite[Fact (i) on 
pp. 80--81]{Sa}, 
building also on ideas of Bernstein and Ginsburg,
will play a fundamental role in the present note. 
We shall assume that $\lambda$ is regular, so that we do not need the assumption
that  sequences are faithfully indexed and, moreover, as well-known, 
in this case,
$[ \lambda, \lambda ]$-compactness is equivalent to the statement that 
every subset of cardinality $\lambda$ has a 
complete accumulation point
($C[ \lambda , \lambda ]$ in Saks' notation). 
See also Caicedo \cite[Section 3]{Ca}, in particular, for variations for the case when $\lambda$ is singular.

\begin{proposition} \labbel{saks}
\cite{Sa} If $\lambda$ is regular, then $X$ is $[ \lambda, \lambda ]$-compact
if and only if, for every sequence $(x_ \alpha ) _{ \alpha \in \lambda } $
of elements of $X$, there is an ultrafilter $D$ uniform over $\lambda$ 
such that $(x_ \alpha ) _{ \alpha \in \lambda } $ $D$-converges to some 
$x \in X$.
 \end{proposition}

\begin{thm} \labbel{produf}
If  $\lambda$ and
$(\mu_ \gamma ) _{ \gamma \in \kappa } $
 are  regular cardinals, then 
$\prod _{ \gamma \in \kappa } \mu _ \gamma   $ is  $[ \lambda, \lambda ]$-compact
if and only if $\lambda {\not\Rightarrow} (\mu_ \gamma ) _{ \gamma \in \kappa } $. 
 \end{thm}

\begin{proof}
Let $X=\prod _{ \gamma \in \kappa } \mu _ \gamma $, and, for 
$ \gamma \in \kappa$, let 
$\pi_ \gamma : X \to \mu_ \gamma $ be the natural projection.
A sequence of functions as in the first line of (*) 
can be naturally identified with a sequence
$(x_ \alpha ) _{ \alpha \in \lambda } $ of elements 
of $X $, by posing
$\pi _ \gamma (x_ \alpha)=f_ \gamma ( \alpha )  $. 
By Proposition \ref{saks},
$X$
is $[ \lambda, \lambda ]$-compact if and only if, 
for every  sequence $(x_ \alpha ) _{ \alpha \in \lambda } $
of elements of $X$, there is an ultrafilter $D$ uniform over $\lambda$ 
such that $(x_ \alpha ) _{ \alpha \in \lambda } $ $D$-converges in $ X$.
As well known, this happens if and only if,
 for each $\gamma \in \kappa $, 
$ (\pi _ \gamma (x_ \alpha)) _{ \alpha \in \lambda } $ 
$D$-converges in $\mu_ \gamma $, and
this happens if and only if,   for each $\gamma \in \kappa $,
there is $\delta_ \gamma \in \mu _ \gamma $
such that 
$\{ \alpha \in \lambda \mid \pi _ \gamma (x_ \alpha) < \delta _ \gamma \} \in D$.
Under the  mentioned identification, and since 
every $\mu_ \gamma $ is regular, 
this
means exactly that  each  $f_ \gamma (D)$ fails to be  uniform over $\mu_ \gamma $.
 \end{proof}

We now consider models of the form 
$\m A = \langle \lambda, <, \alpha , \dots  \rangle _{ \alpha \in \lambda } $
(here, by abuse of notation, we do not distinguish between a symbol
and its interpretation).
If $\m B \equiv \m A$ (that is, $\m B$ is \emph{elementarily equivalent}
to $\m A$), we say that $b \in B$ is \emph{$\lambda$-nonstandard} 
if $ \alpha < b$ holds in $\m B$, for every $\alpha \in \lambda $.
Similarly, for $\mu < \lambda $, we say that    
$c \in B$ is \emph{$\mu$-nonstandard} 
if $  c < \mu $ and $\beta < c $ hold  in $\m B$, for every $\beta \in \mu$.
Of course, in the case $\lambda= \omega $,
we get the usual notion of a nonstandard element.
The importance of $\lambda$-nonstandard elements 
in Model Theory
has been stressed 
by C. C. Chang and H. J. Keisler; 
see \cite[pp. 116--118]{Cha}.
(About the terminology: a $\mu$-nonstandard element $c$ in the 
above sense is said to \emph{realize $\mu$} in  
 \cite{Cha}, and 
a model with a $\mu $-nonstandard element is said  
to \emph{bound $\mu$} in \cite{bumi}.)

\begin{thm} \labbel{nonst}
If $\mu \leq \lambda $ are regular cardinals and 
$\kappa \geq \lambda $,  then 
$\lambda \stackrel{ \kappa }{\not\Rightarrow} \mu$
if and only if, for 
every expansion $\m A$ of 
$\langle \lambda, <, \alpha \rangle _{ \alpha \in \lambda } $
with at most $\kappa$ new symbols
(equivalently, symbols and sorts), there is 
$\m B \equiv \m A$ 
such that $\m B$ has a $\lambda$-nonstandard element
but no $\mu$-nonstandard element.
 \end{thm}

 \begin{proof} 
Suppose $\lambda \stackrel{ \kappa }{\not\Rightarrow} \mu$
and let $\m A$ be an expansion of 
$\langle \lambda, <, \alpha \rangle _{ \alpha \in \lambda } $
with at most $\kappa$ new symbols and sorts.
Without loss of generality, we may assume that 
$\m A$ has Skolem functions, 
since this adds at most $\kappa \geq \lambda $ new symbols.
Enumerate as $(f_ \gamma  ) _{ \gamma \in \kappa }$
all the functions
from  $\lambda $ to $  \mu$
which are definable in $\m A$
(repeat occurrences, if necessary), 
and let $D$ be the ultrafilter given by 
 $\lambda \stackrel{ \kappa }{\not\Rightarrow} \mu$.
Let $\m C$ be the ultrapower $\prod _D \m A$.
Since $D$ is uniform over $\lambda$, 
$b = [Id]_D$, the $D$-class of the identity on $\lambda$,
 is a $\lambda$-nonstandard element in $\m C$.
 Let $\m B$ be the Skolem hull of 
$ \{ b \} $ in $\m C$; thus $\m B \equiv \m C \equiv \m A$,
 and $b$ is a $\lambda$-nonstandard element of $\m B$. 
Had $\m B$ a $\mu$-nonstandard element $c$,
there  would be $\gamma \in \kappa $ such that $c= f_ \gamma (b)$,
by the definition of $\m B$.  Thus 
$c= f_ \gamma ([Id]_D)= [f_ \gamma ]_D$,
but this would imply
that $f_ \gamma (D)$ is uniform over $ \mu$ (since $\mu$
is regular),
contradicting the choice of $D$.  

For the converse, suppose that $(f_ \gamma  ) _{ \gamma \in \kappa }$
is a sequence of functions from $\lambda$ to $\mu$.
Let $\m A$ be the expansion of 
$\langle \lambda, < , \alpha \rangle _{ \alpha \in \lambda } $
 obtained  
by adding the $f_ \gamma $'s as unary functions.
By assumption, there is $ \m B \equiv  \m A$ 
 with a $\lambda$-nonstandard element $b$ 
but without  $\mu$-nonstandard elements.
For every formula $\varphi(y)$ in the similarity type of
$\m A$ and
with exactly one free variable $y$, 
let $Z_ \varphi = \{ \alpha \in \lambda \mid \varphi( \alpha ) \text{ holds in } \m A\}$.
Put 
$E=\{ Z_ \varphi \mid \varphi \text{ is as above, and } 
\varphi(b) \text{ holds in  } \m B \}$. 
$E$ has trivially the finite intersection property, thus it can be extended to some ultrafilter 
$D$ over $\lambda$. Since $\lambda$ is regular and, for every
$\alpha \in \lambda$, $( \alpha, \lambda ) \in E \subseteq D$, we get that $D$ 
is uniform.  
Let $\gamma \in \kappa $. Since  $\m B$ has no 
 $\mu$-nonstandard element,
there is $\beta < \mu$ such that 
$  f_ \gamma (b)< \beta  $  holds in  $  \m B$.  
Letting $\varphi ( y)$
be $  f_ \gamma (y)< \beta  $, we get that
 $Z_ \varphi = 
\{ \alpha \in \lambda \mid  f_ \gamma ( \alpha ) < \beta \}
\in E \subseteq D$, 
proving 
that 
$  f_ \gamma (D)$ is not uniform over $\mu$.
\end{proof} 

If $\Sigma$  and $\Gamma$ are sets of sentences of 
$\mathcal L _{ \omega _1, \omega } $,
we say that $\Gamma$ is \emph{$\mu$-satisfiable
relative to $\Sigma$} if
$\Sigma \cup \Gamma'$
is satisfiable, for every 
$\Gamma' \subseteq \Gamma $ 
of cardinality $<\mu$.
If $\mu \leq \lambda $, we say that $\mathcal L _{ \omega _1, \omega } $ is  
\emph{$ \kappa  $-$( \lambda   , \mu  )$-compact}
if $\Sigma \cup \Gamma$
is satisfiable, whenever 
$|\Sigma| \leq \kappa  $, $|  \Gamma| \leq \lambda $,
and $\Gamma$ is $\mu$-satisfiable
relative to $\Sigma$.
The notion of $ \kappa  $-$( \lambda , \mu )$-compactness
 has been introduced in \cite{bumi} 
for arbitrary logics,  extending  notions by
Chang, Keisler, Makowsky, Shelah and Tarski and others.
Clearly, if $ \kappa \leq \lambda $,
then $ \kappa  $-$( \lambda , \mu   )$-compactness reduces to
the classical notion of $( \lambda , \mu   )$-compactness.
Notice the reversed order of the cardinal parameters with respect to 
the corresponding topological property.

\begin{thm} \labbel{log}  
If $\kappa \geq \lambda $ and $\lambda$ is regular, 
the following conditions are equivalent.
 \begin{enumerate} 
  \item  
$ \omega^ \kappa $ is $[ \lambda , \lambda  ]$-compact.
\item
The language $\mathcal L _{ \omega _1, \omega } $
is $ \kappa  $-$( \lambda , \lambda  )$-compact.
  \item
$\lambda \stackrel{ \kappa }{\not\Rightarrow} \omega $.
 \end{enumerate} 

In particular, if $\lambda$ is regular, then $ \omega^ \lambda $
is finally $\lambda$-compact if and only if 
$\mathcal L _{ \omega _1, \omega } $
is $ ( \lambda , \lambda  )$-compact.
\end{thm}

\begin{proof}
The equivalence of (1) and (3) is the particular case of 
Theorem \ref{produf} when all $\mu_ \gamma $'s equal $ \omega$.  
In view of Theorem \ref{nonst}, it is enough to prove
that (2) is equivalent to the necessary and sufficient condition given there for 
$\lambda \stackrel{ \kappa }{\not\Rightarrow} \omega $.
This is Theorem  3.12 in \cite{bumi} and, anyway, it is a standard argument.
We sketch a proof for the non trivial direction. So, suppose that the condition
in Theorem \ref{nonst} holds. For models without  $ \omega$-nonstandard elements,
a formula of $\mathcal L _{ \omega _1, \omega } $ of 
the form $\bigwedge _{n \in \omega } \varphi_n (\bar{x})  $ 
is equivalent to $\forall y < \omega R(y, \bar{x})$, for a newly introduced relation 
$R$ such that $R(n, \bar{x}) \Leftrightarrow \varphi_n (\bar{x}) $, for every 
$ n \in \omega$. Thus, working within such   models, 
and appropriately extending the vocabulary,
 we may assume that 
 $\Sigma  $ 
and $\Gamma $ are sets of first order sentences. 
If $|\Sigma| \leq \kappa  $, 
and $\Gamma = \{ \gamma _ \alpha \mid \alpha \in \lambda  \} $ is $\lambda$-satisfiable relative to $\Sigma$, 
construct a model $\m A$ which contains
$\langle \lambda, < , \alpha \rangle _{ \alpha \in \lambda } $, and with a relation
$S$ such that, for every $\beta <\lambda $,  $\{ z \in A \mid S( \beta , z) \}$ 
models $\Sigma \cup \{ \gamma _ \alpha \mid \alpha < \beta  \} $.
This is possible, since $\Gamma$ is $\lambda$-satisfiable
relative to $\Sigma$.
If  $\m B \equiv \m A$ is given by
$\lambda \stackrel{ \kappa }{\not\Rightarrow} \omega $,
 and $b \in B$ is $\lambda$-nonstandard, then  
$\{ z \in B \mid S( b , z) \}$ models $\Sigma \cup \Gamma$.

The last statement follows from the trivial fact that
$ \omega^ \lambda $
is finally $\lambda^+$-compact, 
since it has a base of cardinality $\lambda$;
hence $ \omega^ \lambda $ is
 finally  $\lambda$-compact if and only if it is $ [\lambda , \lambda  ]$-compact.
\end{proof}

The assumption that $\lambda$ is regular in Theorem \ref{log}
is only for simplicity: we can devise a modified principle, call it  
$(\lambda, \lambda ) \stackrel{ \kappa }{\not\Rightarrow} \omega $,
which involves $(\lambda, \lambda )$-regular ultrafilters \cite{arch}, and
functions $f_ \gamma : [ \lambda ] ^{< \lambda } \to \omega  $. All the arguments carry over
to get a result corresponding to Theorem \ref{log}. In particular,
the equivalence of (1) and (2) holds with no regularity assumption  on $\lambda$.  
To keep this note within the limits of a reasonable length, we shall present details elsewhere.

\begin{corollary} \labbel{interv}
If $\kappa \geq \lambda $, then
$ \omega^ \kappa $ is finally $ \lambda  $-compact
if and only if  $\mathcal L _{ \omega _1, \omega } $
is $( \kappa  , \lambda  )$-compact.
 \end{corollary}

\begin{proof}
Since $ \omega^ \kappa $ is finally $ \kappa ^+ $-compact,
we have that
it is finally $ \lambda  $-compact
if and only if it is 
 $[ \lambda' , \lambda'  ]$-compact,
for every $\lambda'$ such that 
$\lambda \leq \lambda ' \leq \kappa $.
By Theorem \ref{log} and the preceding remark,
this holds if and only if   
$\mathcal L _{ \omega _1, \omega } $
is $( \lambda'  , \lambda'  )$-compact,
for every $\lambda'$ such that 
$\lambda \leq \lambda ' \leq \kappa $.
It is a standard argument to show that 
this is equivalent to $( \kappa  , \lambda  )$-compactness of $\mathcal L _{ \omega _1, \omega } $. See,
e.~g., \cite[Proposition 2.2(iv)]{bumi}.
\end{proof}

A remark is in order here, about the principle $\lambda \stackrel{ \kappa }{\not\Rightarrow} \mu $. Since there are $\mu^\lambda$ functions 
from  $\lambda$ to $\mu$, we get that if $\kappa, \kappa ' \geq \mu^\lambda$,
then  $\lambda \stackrel{ \kappa }{\not\Rightarrow} \mu $ is equivalent
to $\lambda \stackrel{ \kappa' }{\not\Rightarrow} \mu $,
and it is also equivalent to the statement ``there is some ultrafilter $D$ uniform over $\lambda$ such that, for no function $f: \lambda \to \mu$, $f(D)$ is uniform over
$\mu$''. This property has been widely studied by set theorists,
generally under the terminology ``$D$ over $\lambda$ is $\mu$-indecomposable''. 
In this sense, the particular case $\mu= \omega $ considered in Theorem \ref{log}
 incorporates some results involving measurable and related cardinals.  For example,
if $\lambda$ is regular, 
all powers of $ \omega$ are 
$ [\lambda , \lambda  ]$-compact if and only if 
$ \omega ^{2 ^ \lambda } $ is $ [ \lambda , \lambda  ]$-compact,
if and only if 
$\lambda$   carries some $ \omega_1$-complete uniform ultrafilter
(due to the special property of the cardinal $ \omega_1$,
to the effect that $ \omega_1$-completeness is equivalent to
$ \omega$-indecomposability).
In particular,
we get a classical result by \L o\'s \cite{L}, asserting that
$ \omega ^{2^ \lambda } $ is not finally $ \lambda $-compact, provided 
that $ \lambda $ is regular and there is no measurable cardinal $\leq \lambda $.
Moreover, we get that
  all powers of $ \omega$ are finally $\lambda$-compact if and only if, for
every $\lambda' \geq \lambda$,
there is a $( \lambda ', \lambda ') $-regular   
$ \omega_1$-complete ultrafilter
(in particular, this holds
if $\lambda$ is strongly compact).

Many results about $\mu$-indecomposable ultrafilters over $\lambda$ 
generalize to properties of $\lambda \stackrel{ \kappa }{\not\Rightarrow} \mu $,
for appropriate $\kappa< \mu^ \lambda $, but usually with more involved proofs. We initiated this project in \cite{bumi,arch}. 
Applications to powers of $ \omega$ are presented in the next two corollaries.
Notice that in \cite{bumi} the definition of $\lambda \stackrel{ \kappa }{\not\Rightarrow} \mu$ is given directly by means of the condition in Theorem \ref{nonst}. 
The two definitions do not necessarily coincide for $\kappa < \lambda $;
however, here $\kappa \geq \lambda $ is always assumed.

\begin{corollary} \labbel{umi}
Let $\kappa$ be given, and suppose 
that there is some $\lambda \leq \kappa $ such that 
$ \omega^ \kappa $ is $[ \lambda , \lambda ]$-compact.
If $\lambda$ is the first such cardinal, then 
$\mathcal L _{ \lambda , \omega } $ is 
$ \kappa  $-$( \lambda , \lambda  )$-compact; in particular, 
$\lambda$ is weakly inaccessible (actually, rather high 
in the weak Mahlo hierarchy).   
If, in addition, $2 ^{<\lambda} \leq \kappa $, then 
$\lambda$ is weakly compact; and if  
$2 ^{\lambda} \leq \kappa $, then 
$\lambda$ is measurable.
 \end{corollary}

 \begin{proof} 
From Theorem \ref{log} (1) $\Leftrightarrow $  (2) and 
Theorem 3.9
in \cite{bumi}, applied in the particular case of 
$N=\mathcal L _{ \omega _1, \omega } $.
\end{proof}

As a consequence of Theorem \ref{log} and of Corollary \ref{umi},
if there is no measurable cardinal and
the Generalized Continuum Hypothesis holds,
then $ \omega^ \kappa  $ 
is finally $ \kappa  $-compact if and only if $ \kappa $ is weakly compact;
moreover, 
 $ \omega ^ \kappa $ is never $[ \lambda , \lambda  ]$-compact, for $ \lambda < \kappa $
(only special consequences of GCH are needed  in the above statements: we need 
only that every weakly Mahlo cardinal is inaccessible, and that GCH holds 
at weakly Mahlo cardinals).
The assumptions are necessary:
if $\mu$ is $\mu^+$-compact, then there 
is an $ \omega_1$-complete ultrafilter uniform over $\mu^+$,
hence, by a previous remark, all powers of $ \omega$ are $[ \mu^+ , \mu^+  ]$-compact, hence $ \omega  ^{ \mu ^+} $ is finally $\mu ^+$-compact; 
however, $\mu^+$ is not weakly compact.
Moreover, if $\lambda$ is measurable,
then all powers of $ \omega$ are 
$[ \lambda , \lambda  ]$-compact.
With less stringent large cardinal assumptions, Boos \cite{B}, extending results by 
 Kunen, Solovay and others,
constructed models in which   
GCH fails and $\mathcal L _{ \lambda  , \omega } $ 
(hence also  $\mathcal L _{ \omega _1, \omega } $)  
are  $( \lambda , \lambda  )$-compact but 
$\lambda$  is not weakly compact, not even inaccessible.

For $\mu$, $\lambda$ regular cardinals, the principle
$E_ \lambda ^ \mu  $ asserts that $ \lambda $ has
a nonreflecting stationary set consisting of ordinals
of cofinality $\mu$. The next corollary 
applies not only to powers of $ \omega$, but also to powers 
of regular cardinals (always endowed with the order topology).

\begin{corollary} \labbel{ekl} 
If 
$\mu < \lambda $ are regular, and
$E_ \lambda ^ \mu  $, then
$\mu^ \lambda $ is not $[ \lambda , \lambda  ]$-compact.

If $\Box_ \lambda $, then 
$\mu^{ \lambda^+} $ is not $[ \lambda^+ , \lambda^+  ]$-compact,
for every regular $\mu \leq \lambda $.
\end{corollary}

\begin{proof}
By \cite[Theorem 4.1]{bumi}, if
$E_ \lambda ^ \mu  $, then, in the present notation, 
$\lambda \stackrel{ \lambda  }{\Rightarrow} \mu $  
(this was denoted by $\lambda \stackrel{}{\Rightarrow} \mu $
in \cite{bumi}, a notation not consistent with the present one).
The first statement is immediate from 
Theorem \ref{produf}.
The second statement follows from the well known
fact that $\Box_ \lambda $ implies 
 $E _{ \lambda^+} ^ \mu  $,
for every regular $\mu< \lambda $. 
(We need not bother with the case 
$\lambda= \omega $, since 
 $E _{ \omega _1} ^ \omega   $
is a theorem in ZFC.)
 \end{proof}

Mycielski \cite{M} has also considered the property that 
$ \omega^ \kappa $ contains a closed discrete subset of cardinality $\kappa$.
Clearly, if  this is the case, then  $ \omega^ \kappa $ is not $\kappa$-finally compact,
not even $[ \kappa  , \kappa  ]$-compact,
and not $[ \kappa'  , \kappa'  ]$-compact, for every $\kappa' \leq \kappa $.
 A variation on the methods of the 
present note can be used to show that if $\lambda' \leq \kappa $, then 
$ \omega^ \kappa $ contains a closed discrete subset of cardinality $ \lambda' $
if and only if there is no $\lambda \leq \lambda ' $ 
such that  $\mathcal L _{ \omega _1, \omega } $ is
$ \kappa  $-$( \lambda , \lambda  )$-compact,
 if and only if (by Corollary \ref{umi})
there is no $\lambda \leq \lambda ' $ 
such that 
$\mathcal L _{ \lambda , \omega } $ is 
$ \kappa  $-$( \lambda , \lambda  )$-compact,
if and only if  (by Theorem \ref{log}) for
no $\lambda \leq \lambda ' $
$ \omega^ \kappa $ 
 is $[ \lambda , \lambda  ]$-compact.

Finally, let us notice that, though we have stated our results in terms of powers of 
$ \omega$, they can be reformulated in a way which involves  arbitrary
$T_1$ spaces.

\begin{proposition} \labbel{tutt}
For given $\lambda$ and $\kappa$, the following conditions are equivalent. 
\begin{enumerate}
   \item
$ \omega^ \kappa $ is not  $ [\lambda , \lambda  ]$-compact.
\item  
For every product $X= \prod _{i \in I} X_i$ 
of $T_1$ topological spaces, if $X$ is 
$ [\lambda , \lambda  ]$-compact, then
$|\{ i \in I \mid X_i \text{ is not  countably compact} \}|< \kappa $. 
 \end{enumerate} 
 \end{proposition}

\begin{proof}
(2) $\Rightarrow $  (1) is trivial.
For the converse, notice that if  a $T_1$ topological spaces 
is not countably compact, then it contains a countable discrete closed subset,
that is, a closed copy of $ \omega$; now, use the fact that 
$ [\lambda , \lambda  ]$-compactness is closed-hereditary and preserved 
under surjective homomorphic images.
 \end{proof}

\end{document}